\documentclass[reqno]{amsart}

\usepackage{tikz}
\usepackage{placeins}

\usepackage{graphicx} 

\usepackage{amssymb,amsmath,amsfonts,amsthm}

\usepackage{hyperref}

\usepackage[backend=bibtex, 
style=numeric,
sortcites=true,
giveninits=true,
hyperref,
maxbibnames=99,
doi=false,isbn=false,url=false,eprint=false,
]{biblatex}

\bibliography{biblio}{}

\newtheorem{theorem}{Theorem}[section]
\newtheorem{lemma}[theorem]{Lemma}
\newtheorem{proposition}[theorem]{Proposition}
\newtheorem{corollary}[theorem]{Corollary}
\newtheorem{remark}[theorem]{Remark} 
\newtheorem{example}[theorem]{Example}

\newtheorem{definition}[theorem]{Definition}

\newcommand{\R}{\mathbb{R}}

\newcommand{\C}{\mathbb{C}}
\renewcommand{\Re}{\textrm{Re}}

\newcommand{\area}{\textrm{Area}}
\newcommand{\vol}{\textrm{Vol}}
\newcommand{\capa}{\mathrm{cap}}
\newcommand{\Reff}{R_\textrm{eff}}
\newcommand{\supp}{\textrm{supp}}

\usepackage[svgnames]{xcolor}

\title{A generalized Cheeger inequality and the Steklov Problem on finite graphs}

\author[Hua]{Bobo Hua}
\address{School of Mathematical Sciences, LMNS, Fudan University, Shanghai, China}
\email{bobohua@fudan.edu.cn}

\author[Kamtue]{Supanat Kamtue}
\address{Department of Mathematics
and Computer Science, Faculty of Science, Chulalongkorn University, Bangkok, Thailand}
\email{supanat.k@chula.ac.th}

\author[Liu]{Shiping Liu}
\address{School of Mathematical Sciences, University of Science and Technology of China, Hefei, China}
\email{spliu@ustc.edu.cn}

\author[M\"unch]{Florentin M\"unch}
\address{Institute of Mathematics, Leipzig University, Leipzig, Germany}

\email{cfmuench@gmail.com}
\author[Peyerimhoff]{Norbert Peyerimhoff}
\address{Department of Mathematical Sciences, Durham University, Durham, UK}
\email{norbert.peyerimhoff@durham.ac.uk}

\date{\today}

\begin{document}

\begin{abstract} 

{We  prove generalized Cheeger inequalities for eigenvalues of Laplacians for reversible Markov chains. Then we apply Hassannezhad and Miclo's convergence result to obtain Jammes Cheeger inequalities for Steklov eigenvalues. In particular, we get a sharp estimate for the first non-trivial Steklov eigenvalue via Escobar Cheeger constant. At the end, we extend Hassannezhad and Miclo's convergence result to non-reversible Markov chains via a different method based on resolvent convergence, answering one of their questions.}
\end{abstract}

\maketitle

\tableofcontents

\section{Introduction}
Cheeger \cite{cheeger70}  established an eigenvalue estimate for the Laplacian on Riemannian manifolds in terms of the isoperimetric constant, now commonly referred to as the Cheeger constant. For a closed connected Riemannian manifold $M,$ 
$$\lambda_2(M)\geq \frac14 h_M^2,$$ where $\lambda_2(M)$ is the second eigenvalue of the Laplacian on $M$ (the first eigenvalue is zero) and $h_M$ is the Cheeger constant defined as
$$h_M:=\inf_{\substack{
A\subset M\\ \vol(A)\leq \frac12\vol(M)}}\frac{\area(\partial A)}{\vol(A)},$$
where $A$ runs through all open subsets of $M$ with smooth boundaries. Note that the Cheeger constant is an eigenvalue of the 1-Laplacian operator. The Cheeger estimate plays an important role in spectral geometry in the literature; see, e.g., \cite{buser80,Chavel84,ledoux1994simple,Kawohl03}.

For a compact connected Riemannian manifold with smooth boundary $M,$ the Dirichlet-to-Neumann operator is defined as $$\Lambda_M: H^{\frac12}(\partial M)\to H^{-\frac12}(\partial M),f \mapsto \frac{\partial u_f}{\partial n},$$ where $u_f$ is the harmonic extension of $f$ to $M$ and $n$ is the outward normal vector on $\partial M.$ Note that $\Lambda_M$ is a non-local pseudo-differential operator, whose eigenvalues are called Steklov eigenvalues, which were  introduced by Steklov in 1902; see, e.g., \cite{Steklov1}.
Escobar \cite{escobar97} introduced the so-called Escobar Cheeger constant
\begin{align}\label{eq:escobar}
    h_E(M) &=\inf_{\substack{
    A\subset M\\ \area(A\cap \partial M)\leq \frac12 \area(\partial M)}}\frac{\area(\partial A\cap \mathrm{int}(M))}{\area(A\cap \partial M)},
\end{align}
and proved a Cheeger type estimate of the second Steklov
eigenvalue: for any $a,k>0,$
\begin{align}\label{eq:EscobarEstimate}
    \sigma_2(M) \ge \frac{(h_E(M)\mu_1(k) -ak )a}{a^2+\mu_1(k)},
\end{align} where $\sigma_2(M)$ is the second Steklov eigenvalue of $M$ (the first eigenvalue is zero) and $\mu_1(k)$ is the  Laplacian eigenvalue with the Robin boundary condition $\frac{\partial u}{\partial n}+ku =0\ \text{on } \partial M.$ Jammes \cite{jammes15} introduced the so-called Jammes Cheeger constant
\begin{align}\label{eq:Jammes}
    h_J(M) =\inf_{\substack{
    A\subset M\\ \vol(A)\leq \frac12 \vol(M)}}\frac{\area(\partial A\cap \mathrm{int}(M))}{\area(A\cap \partial M)},
\end{align} and proved another estimate of the second Steklov
eigenvalue:
\begin{align}\label{eq:JammesEstimate}
    \sigma_2(M) \ge \frac{1}{4} h(M)h_J(M),
\end{align} where $h(M)$ is the Cheeger constant of the Laplacian with Neumann boundary condition. See  the survey articles \cite{GP2017,CGGS2024} on these developments.


The analysis of graphs has attracted considerable attention in recent years. Our particular focus are eigenvalue problems on graphs. The Cheeger estimate was extended to graphs by Dodziuk \cite{Dodziuk1984} and Alon and Milman \cite{AM1985}, independently. There has been extensive research on Cheeger estimates for graphs; see, e.g., \cite{Mohar1989,Mohar1991,Chung1997,Chung2010,KellerPeyerimhoff11,LGT2014,BKW2015}. In this paper, we will prove generalized Cheeger inequalities for eigenvalues of Laplacians for reversible Markov chains; see Theorem~\ref{t:gcheeger} below.


The Steklov eigenvalues on graphs with boundary were introduced in \cite{CGB18,HHW2017,hassannezhad-miclo20}.  A discrete analogue of the Escobar Cheeger estimate was established in \cite{HHW2017}, and a discrete analogue of the Jammes Cheeger estimate was proven in \cite{HHW2017,hassannezhad-miclo20}.
Moreover, Hassannezhad and Miclo \cite{hassannezhad-miclo20} proved higher order Cheeger estimates for the Steklov eigenvalues.
See \cite{Perrin2019,Perrin2021,HeHua21,HH2022,HHW2022,Tschanz22,ShiYu22,ShiYu2022,YuYu2022,HH2023,Tschanz23,YuYu24,LinZhao2024,LinZhao25,ShiYu2025} for many other developments on discrete Steklov eigenvalues.


Hassannezhad and Miclo \cite{hassannezhad-miclo20} established an interesting relation between the Laplacian and Steklov eigenvalues: as the vertex measure tends to zero, the bounded Laplacian eigenvalues converge to the Steklov eigenvalues, thereby providing a bridge for studying Steklov eigenvalues via their Laplacian counterparts. By combining this convergence result with our Cheeger estimate for the Laplacian, we obtain new Jammes Cheeger inequalities for Steklov eigenvalues; see Theorem~\ref{t:Jammes} below. Furthermore, in the final section of the paper, we establish a convergence result for Steklov eigenvalues in the setting of non-reversible Markov chains, Theorem~\ref{thm:convlapsteklov}, thereby resolving an open problem posed in \cite{hassannezhad-miclo20}.

\subsection{Setup and notations}
Before we present the results of this paper, let us  introduce the relevant notation. Henceforth, $[n]$ denotes the set $\{1,2,\dots,n\}$ for any positive integer $n \in \mathbb{N}$.
Let $G=(V,p)$ be a continuous time Markov chain with a finite set $V$ of  states with $|V| \ge 2$, transition rates given by $p: V \times V \to [0,\infty)$ 
and $\mu: V \to (0,\infty)$ be an invariant probability measure satisfying
$$ \sum_{x \in V} p_{xy} \mu(x) = \mu(y) \sum_{x \in V} p_{yx},\quad \forall y\in V. $$
In contrast to discrete time Markov chains, we do not have any restrictions on the transition rates except for $p_{xx} = 0$ for all $x \in V$.
We usually write $(V,p,\mu)$ as a continuous time Markov chain with an invariant probability measure $\mu.$
A finite Markov chain gives rise to the following Laplacian on the space $C(V,\C) = \{ f: V \to \C \}$:
$$ \Delta f(x) = \sum_{y \in V} p_{xy}(f(x)-f(y)). $$
An enumeration of the states in $V$ yields an identification of 
the Laplacian $\Delta$ with a $n \times n$ matrix, where $n = |V|$, and the invariant probability measure $\mu$, as a column vector with positive entries, lies in the kernel of $\Delta^\top$. In other words, a probability measure $\mu$ is an invariant measure if and only if $\sum_{x \in V} \Delta f(x) \mu(x) = 0$ for all $f \in C(V,\C)$. Moreover, the transition rates induce the following oriented edge set $E^{or}(G)$: We have $(x,y) \in E^{or}(G)$ if and only if $p_{xy} > 0$.
Let $\langle \cdot,\cdot \rangle: C(V,\mathbb{C}) \times C(V,\mathbb{C}) \to \mathbb{C}$ be a Hermitian inner product, given by
$$ \langle f_1,f_2 \rangle_\mu = \langle f_1,f_2 \rangle = \sum_{x \in V} \mu(x) f_1(x) \overline{f_2(x)}. $$
The associated norm is denoted by
$$ \Vert f \Vert_\mu = \Vert f \Vert = \sqrt{\langle f,f \rangle}. $$
The invariant probability measure $\mu$ is an element of the measure space $\mathcal{M}^*(V) = \{ \nu: V \to (0,\infty) \}$, which is a subset of the slightly more general set $\mathcal{M}(V) = \{ \nu: V \to [0,\infty) \}$, containing also measures which may vanish on certain vertices. The support of $\nu \in \mathcal{M}(V)$ is defined as
$$ {\rm{supp}}(\nu) = \{ x \in V: \nu(x) \neq 0 \}. $$
The subset of probability measures on $V$ is defined as
$\mathcal{P}(V) = \{ \nu: V \to [0,\infty): \sum_{x \in V} \nu(x) = 1 \}$, and $1_x \in \mathcal{P}(V)$ denotes the delta-function at $x \in V$, that is, $1_x(y) = 0$ for $y \neq x$ and $1_x(x)=1$.

The \emph{degree} of a state $x \in V$ is defined as
\begin{equation} \label{eq:degvertex} 
{\rm{deg}}_x ={\rm{deg}}(x) = \sum_{y \neq x} p_{xy}\mu(x), 
\end{equation}
and we can think of $\rm{deg}$ as an element of $\mathcal{M}(V)$.

We call $G=(V,p,\mu)$ a \emph{reversible Markov chain}, if
$$ p_{xy} \mu(x) = p_{yx} \mu(y) \quad \textrm{for all $x,y \in V$}. $$
In the reversible case, we have $p_{xy} > 0$ if any only if $p_{yx} > 0$, and they induce an undirected simple graph structure on $G$, where the edge set is given by $E(G) = \{ \{x,y\} \in V \times V: p_{xy} > 0\}$. Moreover, reversibility of a finite Markov chain is equivalent to the symmetry of the Laplacian $\Delta$ with respect to the above inner product, and it has non-negative real eigenvalues $\lambda_j = \lambda_j(\Delta)$, ordered as
$$ 0=\lambda_1 \le \lambda_2 \le \cdots \le \lambda_n, $$
and counted with multiplicity. 
In the non-reversible case, the eigenvalues of $\Delta$ are no longer necessarily real. 

There is an alternative description of reversible finite Markov chains, namely, as finite weighted graphs $G=(V,w,\mu)$ with symmetric edge weights $w: V \times V \to [0,\infty)$ satisfying $w(x,x) = 0$ for $x \in V$ and vertex measure $\mu: V \to (0,\infty)$. These edge weights give rise to transition rates $p_{xy}$ 
via 
$$ p_{xy} = \frac{w(x,y)}{\mu(x)}.$$
In this notation, the (symmetric graph) Laplacian agrees with the Markov chain Laplacian and takes the form
\begin{equation} \label{eq:Deltamu}
\Delta f(x) = \Delta_{w,\mu} f(x) = \frac{1}{\mu(x)} \sum_{y \in V} w(x,y)(f(x)-f(y)). 
\end{equation}
The eigenvalues of $\Delta_{w,\mu}$ are real and non-negative. For subsets $A,B \subset V$ and a function $f \in C(V,\C)$, we define the complement of $A$ as $A^c = V \setminus A$ and
\begin{align*}
w(A,B) = \sum_{x \in A, y \in B} w(x,y), \\
\mu(f) = \sum_{x \in V} \mu(x) f(x), \\
\mu(A) = \mu(1_A) = \sum_{x \in A} \mu(x),
\end{align*}
where $1_A$ is the characteristic function of the set $A$. Having introduced the relevant notation, we can now discuss the main results of this paper.

\subsection{Results}

Our  result is based on the following generalized Cheeger constant.

\begin{definition}[Generalized Cheeger constant] 
\label{def:gencheeger} Let $G=(V,w,\mu)$ be a finite weighted graph with symmetric edge weights $w$ and vertex measure $\mu \in \mathcal{M}^*(V)$. For another given vertex measure $\nu \in \mathcal{M}(V)$, the corresponding \emph{generalized Cheeger constant} is defined as follows:
$$ h(\mu,\nu) = \inf_{\substack{\emptyset \neq A \subset V: \\ \nu(A) \le \nu(V)/2}} \frac{w(A,A^c)}{\mu(A)}. $$
\end{definition}
One easily sees that since for each $A\subset V,$ either $A$ or $A^c$ is allowed, $$ h(\mu,\nu)\leq  h(\mu,\mu).$$

We have the following generalization of the classical Cheeger inequality for the first non-trivial eigenvalue of a symmetric Laplacian.

\begin{theorem}[Generalized Cheeger inequality]\label{t:gcheeger} Let $G=(V,w,\mu)$ be a connected finite weighted graph with symmetric edge weights $w$, vertex measure $\mu \in \mathcal{M}^*(V)$, and vertex degree ${\rm{deg}}$ given in \eqref{eq:degvertex}. 
Then we have, for any $\nu \in \mathcal{M}(V)$,
$$ \lambda_2 \ge \frac{1}{2} \cdot h(\mu,\nu) \cdot h({\rm{deg}},\nu), $$
where $\lambda_2$ is the second smallest eigenvalue of the Laplacian $\Delta_{w,\mu}$, given in \eqref{eq:Deltamu}.
\end{theorem}

This result is proved and further discussed in Section \ref{sec:cheegineq}.

Using this, we will prove a ``Jammes-type'' Cheeger inequality for symmetric Steklov operators via the convergence result of Hassannezhad and Miclo, which we will discuss later; see Theorem~\ref{t:miclo}. To state it, we need to slightly extend our Cheeger constant in Definition \ref{def:gencheeger} to vertex measures $\mu \in \mathcal{M}(V)$ without full support as follows:
$$ h(\mu,\nu) = \inf_{\substack{A \subset V: \mu(A) > 0, \\ \nu(A) \le \nu(V)/2}} \frac{w(A,A^c)}{\mu(A)}. $$
The proof of the following result is given in Section \ref{sec:jammes}.

\begin{theorem}\label{t:Jammes} Let $G=(V,w,\mu)$ be a connected finite weighted graph with symmetric edge weights $w$ and vertex measure $\mu \in \mathcal{M}^*(V)$, and vertex degree $\rm{deg}$ given in \eqref{eq:degvertex}. Let $B \subset V$ and $\mu_B = 1_B \mu \in \mathcal{M}(V)$. 
Then we have, for any $\nu \in \mathcal{M}(V)$,
$$ 2h(\mu_B,\mu_B)\ge\sigma_2 \ge \frac{1}{2} \cdot h(\mu_B,\nu) \cdot h({\rm{deg}},\nu), $$
where $\sigma_2$ is the second smallest eigenvalue of the Steklov operator $T=T^B$, given in Definition \ref{def:steklov} below. 
\end{theorem}

Note that $h(\mu_B,\mu_B)$ is a discrete analogue of the Escobar Cheeger constant $h_E(M)$ in \eqref{eq:escobar} introduced by \cite{escobar97}, where he proved the Cheeger type estimate \eqref{eq:EscobarEstimate} of the first non-trivial Steklov eigenvalue via the Escobar Cheeger constant and the first eigenvalue of the Laplacian with Robin boundary condition; $h(\mu_B,\deg)$ is a discrete analogue of the Jammes Cheeger constant $h_J(M)$ in \eqref{eq:Jammes} introduced by \cite{jammes15}. 
By the Rayleigh quotient characterization of the first non-trivial Steklov eigenvalue, the Escobar Cheeger constant naturally arises as a potential analogue of the Cheeger constant in the Laplacian case, as indicated by the corresponding upper bound for the eigenvalue. While Escobar’s original estimate involves the first Robin Laplacian eigenvalue, Cheeger-type estimates formulated in terms of the Escobar Cheeger constant remain relatively scarce in the literature. As a corollary of Theorem~\ref{t:Jammes}, by choosing $\nu=\mu_B,$ we prove the estimate of the Steklov eigenvalue via the Escobar Cheeger constant.

\begin{corollary}
    \label{c:Escobar} Let $G=(V,w,\mu)$ be a connected finite weighted graph with symmetric edge weights $w$ and vertex measure $\mu \in \mathcal{M}^*(V)$, and vertex degree $\rm{deg}$ given in \eqref{eq:degvertex}. Let $B \subset V$ and $\mu_B = 1_B \mu \in \mathcal{M}(V)$. 
    Then 
$$ 2h(\mu_B,\mu_B)\ge\sigma_2 \ge \frac{1}{2} \cdot h(\mu_B,\mu_B) \cdot h({\rm{deg}},\mu_B). $$
\end{corollary}
\begin{remark}
\begin{enumerate}
\item 
By Example~\ref{ex:12}, we see that the previous estimate is sharp.

\item Our estimate improves the Jammes Cheeger estimate proved in \cite{HHW2017}. Let $G=(V,w,\mu)$ be a connected weighted graph and $B\subset V$ such that $w(B,B)=0,$ i.e. there is no edges between any two vertices in $B.$ This is a general setting for the Steklov problem on a subset; see \cite{HHW2017}. We consider the case $\mu=\deg$ for the normalized Steklov operator. The following Jammes Cheeger estimate was proven in \cite[Theorem~1.3]{HHW2017}:  \begin{equation}\label{eq:oldJammes}
       \sigma_2\geq \frac12h(\deg_B,\deg) h(\deg,\deg),\end{equation} where $h(\deg_B,\deg)$ is the Jammes Cheeger constant introduced in \cite{HHW2017} and $h(\deg,\deg)$ is the Cheeger constant for the normalized Laplacian on $G$. By Corollary~\ref{c:Escobar}, we have the following estimate
   \begin{equation}\label{eq:21} 2h(\deg_B,\deg_B)\ge\sigma_2 \ge \frac{1}{2} \cdot h(\deg_B,\deg_B) \cdot h({\rm{deg}},\deg_B).\end{equation} In Example~\ref{ex:13}, our lower bound in \eqref{eq:21} is sharp and better than that in \eqref{eq:oldJammes}. 
   \end{enumerate}
\end{remark}

To discuss our next results, which hold in the more general case of a non-reversible finite Markov chain $G=(V,p,\mu)$, we need to introduce the Steklov operator $T=T^B$, associated to a subset 
$B \subset V$, which we consider as a \emph{set of boundary vertices}. A standing condition will be that there is a directed path from every vertex in $V$ to some vertex in $B$ along directed edges in $E^{or}(G)$. We will refer to this condition simply by saying that ``$V$ is connected to $B$''. The following fact is of crucial importance for the well-definedness of the Steklov operator, given in Defintion \ref{def:steklov} below.

\begin{theorem}[Dirichlet Problem]
\label{thm:Dirich-Prob}
Let $G=(V,p,\mu)$ be a (possibly non-reversible) finite Markov chain and $B \subset V$. Assume that $V$ is connected to $B$. Then for any $f \in C(B,\mathbb{C})$ there exists a unique $F \in C(V,\mathbb{C})$ satisfying
$$
\begin{cases}
    \Delta F(x) = 0 & \textrm{for all $x \in B^c$,} \\
    F(x) = f(x) & \textrm{for all $x \in B$.}
\end{cases}
$$
The function $F$ is called the \emph{(unique) harmonic extension} of $f$. 
\end{theorem}

This result gives rise to a well-defined operator ${\rm{Ext}}: C(B,\C) \to C(V,\C)$, where ${\rm{Ext}}(f)$ is the harmonic extension of $f$, and the Steklov operator is defined as follows:

\begin{definition}[Steklov operator]
\label{def:steklov}
Let $G=(V,p,\mu)$ be a (possibly non-reversible) Markov chain and $B \subset V$ a subset such that $V$ is connected to $B$. Then the \emph{Steklov operator} $T=T^B: C(B,\C) \to C(B,\C)$ is given by $T = \Delta \circ {\rm{Ext}}$, that is,
$$ T f(x) = \Delta F (x) \quad \text{for $x \in B$}, $$
where $F$ is the harmonic extension of $f$.
\end{definition}

Similarly as in the case of the Laplacian, the eigenvalues $\sigma_j = \sigma_j(T)$ of the Steklov operator are real-valued and non-negative in the reversible case, and they can be ordered as
$$ 0 = \sigma_1 \le \sigma_2 \le \cdots \le \sigma_b, \quad b = |B|,$$
and counted with multiplicity. Again, in the non-reversible case, these eigenvalues are no longer necessarily real. 

Our next result is an extension of a result by Hassannezhad and Miclo \cite{hassannezhad-miclo20} to the non-symmetric case and  states a particular limit behaviour of the complex eigenvalues of the operators $\Delta_r = (1_B+r1_{B^c})\Delta$, as $r \to \infty$, which can be viewed as speeding up the Markov chain on the interior vertices $B^c = V \setminus B$ of a finite Markov chain. Some of these eigenvalues have the property that their real parts escape to infinity, while the other $b=|B|$ eigenvalues converge to the eigenvalues of the Steklov operator $T$. In fact, the $n=|V|$ eigenvalues of the operators $\Delta_r$ can be expressed by continuous functions $\lambda_1(r),\dots,\lambda_n(r)$ in the parameter $r > 0$ (without any specific ordering). 
For convenience, we recall the result of Hassannezhad and Miclo about the convergence of eigenvalues in the reversible case.
\begin{theorem}[{\cite[Proposition 3]{hassannezhad-miclo20}}]\label{t:miclo}
    Let $G=(V,p,\mu)$ be a finite 
reversible Markov chain with invariant measure $\mu$, $B \subset V$ and $b=|B|$ and $n=|V|$. Assume $\lambda_1(r)\leq\dots\leq \lambda_n(r).$ Then for any $1\leq k\leq b,$
$$\lim_{r\to\infty}\lambda_k(r)=\sigma_k,$$ and for $k>b,$ $$\lim_{r\to\infty}\lambda_k(r)=\infty.$$
\end{theorem}

It was mentioned in Hassannezhad and Miclo \cite[Remark 4]{hassannezhad-miclo20} that there should be a generalization to the non-reversible case and we provide a positive answer to their remark in the following. To present our precise result for non-reversible Markov chains, we need to introduce the block matrix representation
\begin{equation} \label{eq:Deltablock}
  \Delta = 
  \begin{pmatrix}
    L_{11} & L_{12}\\
    L_{21} & L_{22}
  \end{pmatrix}
\end{equation}
of $\Delta$ with respect to a vertex enumeration starting with the vertices of $B$ and ending with the vertices of $B^c$. Note that $L_{22}$ is a matrix of size $(n-b) \times (n-b)$ and corresponds to the Laplacian on $C(B^c,\C)$ with Dirichlet boundary conditions.

\begin{theorem} \label{thm:convlapsteklov} 
Let $G=(V,p,\mu)$ be a (possibly non-reversible) finite Markov chain, $B \subset V$ and $b=|B|$ and $n=|V|$.
Assume that $V$ is connected to $B$. Let $\lambda_1(r),\dots,\lambda_n(r) \in \C$ be the continuous (in $r$) eigenvalue functions of the operators $\Delta_r = (1_B+r1_{B^c})\Delta$. Then the minimum $\epsilon$ of the real parts of the eigenvalues of $L_{22}$ in \eqref{eq:Deltablock} is strictly positive, and there exists a permutation $\pi: [n] \to [n]$, such that 
$$ \lambda_{j,\infty}= \lim_{r \to \infty} \lambda_{\pi(j)}(r) \in \C $$ 
exist for all $j \in [b]$ and agree with the eigenvalues $\sigma_j$ of the Steklov operator $T$, counted with multiplicity, and that
$$ \liminf_{r \to \infty} \frac{1}{r} {\rm{Re}} \lambda_{\pi(j)}(r) \ge \epsilon \quad \text{for all $j \in \{b+1,\dots,n\}$.} $$
In particular, the real parts of the eigenvalues $\lambda_{\pi(j)}(r)$, $j \ge b+1$, escape to infinity, as $r \to \infty$. 
\end{theorem}

In the proof of the reversible case in  \cite{hassannezhad-miclo20},  the crucial ingredients are that the eigenvalues are real and canonically ordered, and that there is a useful Rayleigh quotient characterization of eigenvalues of the Laplacian.
The novelty of the proof in the non-reversible case, given in Section \ref{sec:stekconv}, is the convergence of the resolvents of operators. Our arguments are based on linear algebra and maximum principles.

\section{Symmetric Steklov operators and Laplacians}

As a warm-up for this section, we express the Steklov operator via effective resistances.
In the subsequent subsections, we prove generalized Cheeger type estimates for the first nontrivial eigenvalue of the Laplacian. Applying this together with the result of Hassannezhad and Miclo, we prove Jammes Cheeger inequalities for the Steklov operator.

\subsection{Steklov operator via effective resistance}

Let $G=(V,w,\mu)$ be a finite weighted graph with symmetric edge weights $w$, and $B \subset V$.
The Steklov operator can be represented as a symmetric graph Laplacian, that is, for $x \in B$,
\begin{equation}
\label{eq:Tflap}
Tf (x) = \frac 1{\mu(x)} \sum_{y \in B} (f(x)-f(y)) w_B(x,y)
\end{equation}
for a suitable symmetric $w_B: B\times B \to [0,\infty) $, see, e.g., \cite{HHW2017,hassannezhad-miclo20}. Later, in Section \ref{sec:maxprinc}, we show that Steklov operators can also be written as Laplacians on the boundary in the non-symmetric setting.

For expressing $w_B$, we will use the effective resistance or the capacity respectively. 
We recall for $X,Y \subset V$, the capacity is defined as
\[
\capa(X,Y) := \inf \{\langle f,\Delta f \rangle_\mu : f|_X=1, f|_Y=0 \}
\]
and $\Reff(X,Y) :=1/\capa(X,Y)$. 
From an electric network perspective, the capacity is the current flowing from $X$ to $Y$, assuming the electric potential $1$ at $X$ and $0$ at $Y$, where each edge $(x,y)$ is a resistor with resistance $1/w(x,y)$.

\begin{theorem} \label{thm:stekreff}
For two distinct vertices $x,y \in B$, we have
\[
2w_B(x,y) = \capa(\{x\},B \setminus \{x\}) + \capa(\{y\},B \setminus \{y\}) - \capa(\{x,y\},B \setminus \{x,y\}).
\]
\end{theorem}

\begin{proof}
We first claim that for all $X \subseteq B$,
\[
\langle T 1_X, 1_X \rangle_{\mu_B} = \capa(X,B\setminus X).
\]
To prove this, let $f$ satisfy $f|_X = 1$ and $f|_{B\setminus X} = 0$ and $\Delta f = 0 $ on $V \setminus B$.
Then on $B$, we have $T1_X = \Delta f$. Moreover, $\langle f, \Delta f\rangle_\mu = \capa(X,B\setminus X)$ as $f$ is a minimizer for the capacity, by a variational principle.
As $\supp (\Delta f) \subseteq B$, we have
\[
\capa(X,B\setminus X)=
\langle f, \Delta f\rangle_\mu = \langle f|_B, (\Delta f)|_B\rangle_{\mu_B}
= \langle 1_X, T1_X\rangle_{\mu_B},
\]
proving our first claim.
We next observe that for two distinct $x,y \in B$, by using \eqref{eq:Tflap},
\[
-\langle T 1_x, 1_y \rangle_{\mu_B} =  w_B(x,y). 
\]
Hence, we can apply the polarization formula to obtain
\begin{align*}
  2w_B(x,y)  &= -2\langle T 1_x, 1_y \rangle_{\mu_B}  \\&= \langle T 1_x, 1_x \rangle_{\mu_B}  + \langle T 1_y, 1_y \rangle_{\mu_B} - \langle T (1_x + 1_y), 1_x + 1_y \rangle_{\mu_B}  \\
  &=\capa(\{x\},B \setminus \{x\}) + \capa(\{y\},B \setminus \{y\}) - \capa(\{x,y\},B \setminus \{x,y\}),
\end{align*}
finishing the proof.
\end{proof}

\begin{corollary} \label{cor:stekrefftwoverts}
    Let $x,y \in V$ be two distinct vertices and $B = \{x,y\}$. Then we have
    $$ \sigma_2 = \left( \frac{1}{\mu(x)} + \frac{1}{\mu(y)} \right) \capa(\{x\},\{y\}). $$
\end{corollary}

\begin{proof}
    It follows from Theorem \ref{thm:stekreff} that
    $$ w_B(x,y) = \capa(\{x\},\{y\}) $$
    since $\capa(\{x,y\},\emptyset) = 0,$ and by \eqref{eq:Tflap}, the Steklov operator $T$ has the following matrix representation:
    $$ T = \begin{pmatrix}  \frac{w_B(x,y)}{\mu(x)} & - \frac{w_B(x,y)}{\mu(x)} \\ -\frac{w_B(x,y)}{\mu(y)} & \frac{w_B(x,y)}{\mu(y)}\end{pmatrix}. $$
    Hence, we have
    $$ \sigma_2 = {\rm{tr}}(T) = \frac{w_B(x,y)}{\mu(x)} + \frac{w_B(x,y)}{\mu(y)}, $$
    which finishes the proof.
\end{proof}

\subsection{A generalized Cheeger inequality for the Laplacian} \label{sec:cheegineq}

For a reversible Markov chain $G=(V,p,\mu)$ and for an additional vertex measure $\nu \in \mathcal{M}(V),$ we recall the generalized Cheeger constant
$$
    h(\mu,\nu) = \inf_{\substack{\emptyset \neq A \subset V: \\ \nu(A) \le \nu(V)/2}} \frac{w(A,A^c)}{\mu(A)}.
$$
With this notion, we recall the generalized Cheeger inequality (Theorem \ref{t:gcheeger}) from the introduction.


\begin{theorem}[Generalized Cheeger inequality] \label{thm:gencheegineq}
Let $G=(V,p,\mu)$ be a connected finite reversible Markov chain 
with edge weights $w: E \to (0,\infty)$ and vertex measure $\mu \in \mathcal{M}^*(V)$. 
Then we have, for any $\nu \in \mathcal{M}(V)$,
$$ \lambda_2(\Delta_{w,\mu}) \ge \frac{1}{2}\cdot h(\mu,\nu) \cdot h(\deg,\nu). $$
\end{theorem}



\begin{remark} In the special case $\mu=\nu\ge{\rm{deg}}$, we have the following standard Cheeger estimate
\begin{equation} \label{eq:standcheegineq} 
\lambda_2(\Delta_{w,\mu}) \ge \frac{h_\mu(G)^2}{2} \end{equation}
with
$$ h_\mu(G) = h(\mu,\mu) = \inf_{\substack{\emptyset \neq A \subset V: \\ \mu(A) \le \mu(V)/2}}\frac{w(A,A^c)}{\mu(A)}. $$
Since $h(\mu,\nu) \le h(\deg,\nu)$, inequality \eqref{eq:standcheegineq} follows also from Theorem \ref{thm:gencheegineq}. The improvement of this generalization is illustrated in Example \ref{ex:cheegimprov} below.
The additional flexibility in Theorem \ref{thm:gencheegineq} allows also to concentrate the vertex measure $\mu$ increasingly on a subset $B \subset V$ which, in the limit, provides a connection to the Steklov operator, as we will discuss later.
\end{remark}

\begin{proof}
Let $\Delta = \Delta_{w,\mu}$.
Recall the following expression for the Rayleigh quotient,
$$ R(f) = \frac{\sum_{\{x,y\} \in E} w(x,y)|f(y)-f(x)|^2}{\sum_x \mu(x) |f(x)|^2} = \frac{\langle \Delta f, f\rangle_\mu}{\Vert f \Vert_\mu^2},
$$    
and the corresponding variational description of $\lambda_1$:
$$ \lambda_2(\Delta) = \min_{\substack{f \in C(V,\mathbb{R}):\\ \langle f, {\bf{1}} \rangle_\mu =0}} R(f). $$
Let $F \in C(V,\R)$ be a function satisfying
$$ \lambda_2(\Delta) = R(F). $$
Assume, without loss of generality, that we have
$$ \nu(\{x \in V: F(x)>0\}) \le \frac{1}{2} \nu(V). $$
%
Using
\begin{align*} 
\Delta F^+(x) &= \frac{1}{\mu(x)} \sum_{\{x,y\}\in E} w(x,y)(F^+(x)-F^+(y)) \\
&\le \frac{1}{\mu(x)} \sum_{\{x,y\}\in E} (F(x)-F(y)) = \lambda_2(\Delta)F(x)
\end{align*}
for all $x \in V^+ = \{x \in V: F(x) > 0 \}$, we obtain
\begin{align*} 
\lambda_2(\Delta) \Vert F^+ \Vert_\mu^2 &= \lambda_2(\Delta) \sum_{x \in V} \mu(x) (F^+(x))^2 \\
&\ge \sum_{x \in V^+} \mu(x) F^+(x) \Delta F^+(x) = \langle \Delta F^+,F^+ \rangle_\mu, 
\end{align*}
that is,
$$ \lambda_2(\Delta) \ge R(F^+). 
$$
We obtain, using Cauchy-Schwarz in $(*)$ below and
$$ \sqrt{a^2+b^2} \ge \frac{1}{\sqrt{2}}(|a|+|b|) $$
in $(**)$,
\begin{multline*}
R(F^+) = \frac{\langle \Delta F^+,F^+ \rangle_\mu}{\Vert F^+ \Vert^2_\mu} \cdot \frac{\Vert F^+ \Vert^2_{\rm{deg}}}{\Vert F^+ \Vert^2_{\rm{deg}}}
\\ = \frac{\sum_{\{x,y\} \in E} w(x,y)|F^+(y)-F^+(x)|^2 \cdot \sum_{\{x,y\} \in E} w(x,y) (|F^+(x)|^2+|F^+(y)|^2)}{\Vert F^+ \Vert^2_\mu \cdot \Vert F^+ \Vert^2_{\rm{deg}}} \\
\stackrel{(*)}{\ge} \frac{\left(\sum_{\{x,y\} \in E} w(x,y)|F^+(y)-F^+(x)|\sqrt{|F^+(x)|^2+|F+(y)|^2}\right)^2}{\Vert F^+ \Vert^2_\mu \cdot \Vert F^+ \Vert^2_{\rm{deg}}} \\
\stackrel{(**)}{\ge}
\frac{\left(\sum_{\{x,y\} \in E} w(x,y)|F^+(y)-F^+(x)|(F^+(x)+F^+(y))\right)^2}{2\Vert F^+ \Vert^2_\mu \cdot \Vert F^+ \Vert^2_{\rm{deg}}} 
\\ = \frac{\left(\sum_{\{x,y\} \in E} w(x,y)|F^+(y)^2-F^+(x)^2|\right)^2}{2\Vert F^+ \Vert^2_\mu \cdot \Vert F^+ \Vert^2_{\rm{deg}}}.
\end{multline*}
Introducing $H = (F^+)^2$ and $V(t) := \{x \in V: H(x) \ge t \}$ and using the co-area formulas, we have
$$
R(F^+) \ge \frac{1}{2} \frac{\left(\sum_{\{x,y\} \in E} w(x,y)|H(y)-H(x)|\right)^2}{\langle 1,H \rangle_\mu \cdot \langle 1,H \rangle_{\rm{deg}}} = \frac{1}{2} \frac{\left( \int_0^\infty w(V(t),V(t)^c) dt \right)^2}{\int_0^\infty \mu(V(t)) dt \cdot \int_0^\infty \deg(V(t)) dt}.
$$
Note that, for all $t > 0$, 
$$
\nu(V(t)) \le \nu\left(\{x \in V: F(x) > 0\}\right) \le \frac{1}{2}\nu(V). $$ 
Hence, we have 
\begin{align*}
w(V(t),V(t)^c) &\ge \mu(V(t)) h(\mu,\nu), \\ \text{and} \quad w(V(t),V(t)^c) &\ge {\rm{deg}}(V(t)) h({\rm{deg}},\nu)
\end{align*}
for $t > 0$, and therefore
\begin{align*}
\lambda_2(\Delta) \ge R(F^+) &\ge \frac{1}{2} \frac{ \int_0^\infty w(V(t),V(t)^c) dt}{\int_0^\infty \mu(V(t)) dt} \cdot \frac{\int_0^\infty w(V(t),V(t)^c) dt}{\int_0^\infty \deg(V(t)) dt} \\ &\ge \frac{h(\mu,\nu)h({\rm{deg},\nu)}}{2}.
\end{align*}
\end{proof}

\begin{example} \label{ex:cheegimprov} Let $G = (V,E)$ be a path of length $2$, that is, $V=\{x,y,z\}$ and $E=\{ \{x,y\},\{y,z\}\}$. Let $w \equiv 1$ and $\mu \in \mathcal{M}^*(V)$ be given by
$$ \mu(x) = \mu(z) = \frac{1}{\epsilon}  \quad \text{and} \quad \mu(y) = 2, $$
for $\epsilon \in (0,1)$. It is easy to verify that
$$ h_\mu(G) = \frac{w(\{x\},\{y,z\})}{\mu(x)} = \epsilon, $$
and the standard Cheeger inequality implies
$$ \lambda_2(\Delta_{w,\mu}) \ge \frac{\epsilon^2}{2}. $$
The eigenvalues of $\Delta_\mu$ are $0,\epsilon,1+\epsilon$, and therefore,
$$ \lambda_2(\Delta_{w,\mu}) = \epsilon, $$
with the corresponding eigenfunction $f(x)=-f(z)=1$ and $f(y)=0$.
Choosing $\nu=\mu$, we obtain
\begin{align*}
h(\mu,\nu) = h_\mu(G) &= \epsilon, \\
h({\deg},\nu) &= 
\frac{w(\{x\},\{y,z\})}{\deg(x)} = 1,
\end{align*}
and our generalized Cheeger inequality yields the improved inequality
$$ \lambda_2(\Delta_{w,\mu}) \ge \frac{\epsilon}{2}, $$
which is of the right order in $\epsilon$.
\end{example}

\subsection{Jammes Cheeger inequalities for the Steklov operator}
\label{sec:jammes}

Using Hassannezhad and Miclo's limiting argument and the result in the previous subsection, we prove Jammes Cheeger inequalities for the Steklov operator, including the inequalities for the Escobar Cheeger constant.

\begin{proof}[Proof of Theorem~\ref{t:Jammes}]
    We consider the Laplacian $\Delta_r = (1_B+r1_{B^c})\Delta$ for $r>0,$ with the invariant measure $\mu_r:=(1_B+\frac1r1_{B^c})\mu.$ For the lower bound, we use the generalized Cheeger estimate, Theorem~\ref{t:gcheeger}, 
   $$\lambda_2(r) \ge \frac{1}{2} \cdot h(\mu_r,\nu) \cdot h({\rm{deg}},\nu).$$ By passing to the limit, $r\to\infty,$ we have the convergence result of Hassannezhad and Miclo, Theorem~\ref{t:miclo}, 
   $$\lim_{r\to \infty}\lambda_2(r)= \sigma_2.$$ 
   We observe that
$$\frac{1}{h(\mu_r,\nu)}=\max_{A\subset V}1_{\nu (A)\leq \frac12 \nu(V)}\frac{\mu_r(A)}{w(A,A^c)}\to \frac{1}{h(\mu_B,\nu)},\quad r\to\infty.$$
   Therefore, the Cheeger constants are convergent. This implies the lower bound
   $$\sigma_2 \ge \frac{1}{2} \cdot h(\mu_B,\nu) \cdot h({\rm{deg}},\nu).$$

For the upper bound,
$2h(\mu_B,\mu_B)\ge\sigma_2,$ we use the Rayleigh quotient characterization of $\sigma_2,$ i.e., 
$$\sigma_2=\inf_{f\in C(V,\R)}\frac{\langle \Delta f, f\rangle_\mu}{\inf_{c\in \R}\Vert f-c \Vert_{\mu_B}^2}.$$
For a minimizer $A$ of $h(\mu_B,\mu_B),$ we set $f=1_A.$ 
Note that $\langle \Delta f, f\rangle_\mu=w(A,A^c).$ Moreover, for all $c\in \R,$ using the fact that $\mu(B)\ge 2\mu_B(A),$
\begin{align*}
    \Vert f-c \Vert_{\mu_B}^2&=\| f\|_{\mu_B}^2-2\langle f,c\rangle_{\mu_B}+c^2\mu(B) \\
    &= \mu_B(A)- 2c\mu_B(A) + c^2\mu(B)\\
    &\ge \mu_B(A)(1-2c+2c^2) \ge \frac{1}{2}\mu_B(A).
\end{align*} This yields the upper bound. Thus, the proof is finished. 
\end{proof}

Next, we will construct some examples to show the sharpness of our Jammes Cheeger estimates for the Steklov eigenvalues.
\begin{example} \label{ex:12} Let $G = (V,E)$ be a path of length $2$, that is, $V=\{x,y,z\}$ and $E=\{ \{x,y\},\{y,z\}\}$. Let $B=\{x,z\}$, $w \equiv 1$ and $\mu \in \mathcal{M}^*(V)$ be given by
$$ \mu(x) = \epsilon  \quad \text{and} \quad \mu(y)=\mu(z) = 1, $$
for $\epsilon \in (0,1)$. 
One easily shows by Corollary \ref{cor:stekrefftwoverts} that  $\sigma_2=\frac12(1+\frac{1}{\epsilon}).$ Moreover, we have
$$h(\mu_B,\mu_B)=\frac{1}{\epsilon},\quad h(\deg,\mu_B)=\frac{1}{3}.$$ Hence, our estimate is sharp in the order of $\frac{1}{\epsilon}.$

\end{example}

\begin{example} \label{ex:13}
For $n\geq 2,$ let $G=(V,E)$ be a graph with vertex set $$V=\{x_1,\cdots,x_{2n+1},y_1,\cdots, y_n, z_1,z_2\}$$ and $E=\{\{x_i,x_{i+1}\},1\leq i\leq 2n, \{y_j,z_k\}, 1\leq j\leq n, k=1,2, \{x_{2n+1},z_1\}\}.$ See Figure \ref{fig:ex13} for an illustration. Let $B=\{x_1,z_2\}$, $w \equiv 1$ and $\mu=\deg.$ One can verify that the effective resistance  between $x_1$ and $z_2$ is $2n+1+\frac2n$ , i.e., 
$$ \Reff(\{x_1\},\{z_2\}) = 
2n+1+\frac2n. $$
(See \cite{ds84,bar17} for more details.) By Corollary \ref{cor:stekrefftwoverts} (or also \cite[Prop. 2.2]{HHW2017}), 
we obtain, $$\sigma_2=\frac{1}{2n+1+\frac2n}\left(1+\frac1n\right)\sim \frac{1}{2n},\quad n\to\infty.$$ Moreover, we have $$h(\deg_B,\deg_B)=1,\quad h(\deg_B,\deg)=\frac1n,\quad h(\deg,\deg_B)=h(\deg,\deg)=\frac{1}{4n+1}.$$
(Note that the optimal sets attaining the respective Cheeger constants are $A =\{x_1,\dots,x_{2n+1}\}$, $V \setminus A$, $A$, $A$.)
Hence, our lower bound estimate in \eqref{eq:21} is sharp
in the order of $\frac1n,$ which is better than \eqref{eq:oldJammes}.
\end{example}

\begin{figure}
\begin{center}
\begin{tikzpicture}[thick,scale=1]%
\filldraw [black] (0,5) circle (2pt);
\filldraw [black] (1,5) circle (2pt);
\filldraw [black] (3,5) circle (2pt);
\filldraw [black] (4,5) circle (2pt);
\filldraw [black] (5,5) circle (2pt);
\filldraw [black] (5,5) circle (2pt);
\filldraw [black] (7,6.5) circle (2pt);
\filldraw [black] (7,5.5) circle (2pt);
\filldraw [black] (7,3.5) circle (2pt);
\filldraw [black] (9,5) circle (2pt);

\draw (0,5) node[label=above:$x_1$]{};
\draw (1,5) node[label=above:$x_2$]{};
\draw (3,5) node[label=above:$x_{2n}$]{};
\draw (4,5) node[label=above:$x_{2n+1}$]{};
\draw (5,5) node[label=above:$z_{1}$]{};
\draw (7,6.5) node[label=above:$y_{1}$]{};
\draw (7,5.5) node[label=above:$y_{2}$]{};
\draw (7,3.5) node[label=above:$y_{n}$]{};
\draw (9,5) node[label=above:$z_{2}$]{};

\draw (0,5) -- (1,5);
\draw (1,5) -- (1.5,5);
\draw[loosely dotted] (1.5,5) -- (2.5, 5);
\draw (2.5,5) -- (3,5);
\draw (3,5) -- (4,5);
\draw (4,5) -- (5,5);
\draw (5,5) -- (7,6.5);
\draw (5,5) -- (7,5.5);
\draw (5,5) -- (7,3.5);
\draw[loosely dotted] (7,5) -- (7,4.3);
\draw (7,6.5) -- (9,5);
\draw (7,5.5) -- (9,5);
\draw (7,3.5) -- (9,5);
\end{tikzpicture}
\end{center}
\caption{An illustration of the graph in Example \ref{ex:13}. \label{fig:ex13}}
\end{figure}
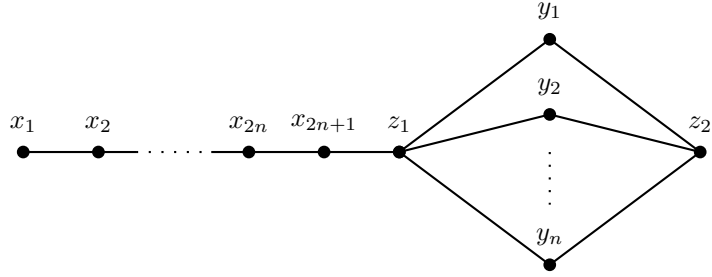

\section{Non-symmetric Steklov operators}
For the non-reversible Markov chain, it is not obvious how to define the Steklov operator due to the lack of Rayleigh quotient characterization. We overcome this difficulty using the maximum principle. Moreover, the maximum principle is the key ingredient for our proof of a Hassannezhad and Milco type convergence result.

\subsection{Maximum Principle and harmonic extensions}
\label{sec:maxprinc}


The results in this section hold for all not necessarily reversible finite Markov chains $G=(V,p,\mu)$, unless stated otherwise. An essential tool to prove the uniqueness of the Dirichlet Problem, formulated in Theorem \ref{thm:Dirich-Prob}, is the Maximum Principle. We will need the following complex-valued local version.

\begin{lemma}[Local Maximum Principle]  \label{lem:neighbourharm}
    Let $f \in C(V,\C)$.
    If $x \in V$ satisfies $|f(x)| = \Vert f \Vert_\infty$ and
    \begin{equation} \label{eq:lapvanish} 
    \Re \left( \overline{f(x)} \, \Delta f(x) \right) \le 0, 
    \end{equation}
    then we have
    $$ |f(y)| = |f(x)| \quad \textrm{for every $y \in V$ with $(x,y) \in E^{or}(G)$.} $$
\end{lemma}

\begin{proof}
    It follows from condition \eqref{eq:lapvanish} that
    $$ 0 \ge {\rm{Re}} \left( \overline{f(x)} \, \Delta f(x) \right) = \sum_{y} p_{xy} \left( \underbrace{| f(x) |^2 - {\rm{Re}} (\overline{f(x)} f(y))}_{\ge 0} \right) \ge 0. $$
    Since $y \in V$ is neighbour of $x$ iff $p_{xy} > 0$, we conclude from this that $\vert f(y) \vert = \vert f(x) \vert$ for every neighbour $y$ of $x$. 
\end{proof}

With this tool at hand, we can show that the Dirichlet Problem has a unique solution.

\begin{proof}[Proof of Theorem \ref{thm:Dirich-Prob}]

Uniqueness of the harmonic extension $F \in C(V,\mathbb{C})$ follows from the Maximum Principle: Assuming $F_1, F_2 \in C(V,\mathbb{C})$ are two solutions of the Dirichlet Problem, then $F = F_1 - F_2$ satisfies
\begin{align*}
    \Delta F(x) &= 0 \quad \textrm{for all $x \in B^c= V \setminus B$,} \\
    F(x) &= 0 \quad \textrm{for all $x \in B$.}
\end{align*}
Let $x \in V$ be a vertex with $|F(x)| = \Vert F \Vert_\infty$. If $x \in B$, we are done. If $x \not\in B$, there exists a directed path from $x$ to some vertex $w \in B$, by assumption, and harmonicity of $F$ on $B^c$ and Lemma \ref{lem:neighbourharm} implies that $|F(x)| = |F(w)| = 0$. This completes the uniqueness proof.

To show existence of an extension, we introduce the operator $\tilde{\Delta}_B \,:\, C(V,\mathbb{C}) \to C(V,\mathbb{C})$ defined as:
$$ \tilde{\Delta}_B f(x) = \begin{cases} \Delta f(x) & \textrm{if $x \in B^c$,} \\
f(x) & \textrm{if $x \in B$.} \end{cases}$$
Note that if $\tilde{\Delta}_B f=0$, then $f=0$. It means $\tilde{\Delta}_B$ is injective and, therefore, surjective. One can check that, given $f\in C(B,\mathbb{C})$, a solution $F \in C(V,\mathbb{C})$
satisfying
\begin{align*}
    \Delta F(x) &= 0 \quad \textrm{for all $x \in B^c$,} \\
    F(x) &= f(x) \quad \textrm{for all $x \in B$,}
\end{align*} 
can be given by $F=\left(\tilde{\Delta}_B\right)^{-1} f_B$, where $f_B \in C(V,\mathbb{C})$ is the extension of $f$ by zero. 
\end{proof}

\begin{remark} \label{rem:locglobMP}
A very similar proof to Lemma \ref{lem:neighbourharm} gives the following local Maximum Principle for real-valued functions:
Let $f \in C(V,\R)$. If $x \in V$ satisfies $f(x) = \Vert f\Vert_\infty$ and
$$ \Delta f(x) \le 0, $$
then we have
$$ f(y) = f(x) \quad \textrm{for every $y \in V$ with $(x,y) \in E^{or}(G)$.} $$
This result implies the following global Maximum Principle for finite Markov chains with boundary $B \subset V$, by using the same arguments as for the uniqueness proof in the Dirichlet Problem: Let $V$ be connected to $B$. If a real-valued function $f \in C(V,\R)$ is harmonic on $B^c$, then it assumes both maximum and minimum at vertices in $B$.
\end{remark}

Assume now that $B \subset V$ is a non-empty subset of boundary vertices and that $V$ is connected to $B$. Recall from the Introduction that we denote the harmonic extension of a function $f \in C(B,\C)$ by ${\rm{Ext}}(f)$. For any $y \in V$, let $\mu_y: B \to \R$ be defined by
$$ \mu_y(x) = {\rm{Ext}}(1_x)(y). $$
It follows from the global Maximum Principle in Remark \ref{rem:locglobMP} that $\mu_y(x) \in [0,1]$ and, by linearity, that the harmonic extension $F \in C(V,\C)$ of $f \in C(B,\C)$ satisfies
$$ F(y) = \mu_y(f) = \sum_{z \in B} f(z) \mu_y(z). $$
Moreover, we have $\mu_y(B) =1$ for all $y \in V$, since $1_V$ is the harmonic extension of $1_B$. Therefore, the family $\{ \mu_y \}_{y \in V}$ lies in $\mathcal{P}(B)$ and is called the set of \emph{harmonic measures} associated to the Laplacian $\Delta$. Probabilistically, $\mu_x(B_0)$ for a subset $B_0 \subset B$, is the probability that the random walk with the transition rates $p_{yz}$, starting at $x$, will hit the boundary $B$ in the set $B_0$ before it hits $B \setminus B_0$. The Steklov operator of a function $f \in C(B,\C)$ at $x \in B$ is then given by
\begin{align*} 
Tf(x) &= \sum_{y \in V} p_{xy} (f(x) - \mu_y(f)) \\
&= \sum_{y \in V} p_{xy} \left( \sum_{z \in B} \mu_y(z) f(x) - \sum_{z \in B} \mu_y(z)f(z) \right) \\
&= \sum_{z \in B}
\underbrace{\left( \sum_{y \in V} \mu_y(z) p_{xy} \right)}_{= \tilde p_{xz}} (f(x) - f(z)).
\end{align*}
We see that
$T$ can also be viewed as a Laplacian on the finite Markov chain $(B,\tilde p,\tilde \mu)$ with $\tilde \mu=\mu_B/\mu(B))$. Note that $\mu_B$ is an invariant measure for the transition rates $\tilde p_{xy}$ since, for any $f \in C(B,\C)$ with harmonic extension $F \in C(V,\C)$, we have
$$ 0 = \sum_{x \in V} \Delta F(x) \mu(x) = \sum_{x \in B} \Delta F(x) \mu_B(x) = \sum_{x \in B} T f(x) \mu_B(x). $$

Of course, the Steklov operator satisfies also the global Maximum Principle mentioned in Remark \ref{rem:locglobMP}.

Let us briefly consider the case when $G=(V,p,\mu)$ is reversible. Let $f,g \in C(B,\R)$ and $F,G \in C(V,\R)$ be their harmonic extensions. Then the Laplacian on $C(V,\R)$ is symmetric and we have
\begin{align*}
\langle Tf,g \rangle_{\tilde \mu} &=\sum_{x \in B} \tilde \mu(x) \Delta F(x) g(x) = \frac{1}{\mu(B)} \langle \Delta F, G \rangle_{\mu} \\
&= \frac{1}{\mu(B)} \langle F, \Delta G \rangle_{\mu} = \sum_{x \in B} \tilde \mu(x) f(x) \Delta G(x) \\
&= \langle f,Tg \rangle_{\tilde \mu},
\end{align*}
which shows that the Steklov operator $T $ on $C(B,\R)$ is also symmetric with respect to the inner product $\langle \cdot,\cdot \rangle_{\tilde \mu}$. Moreover, the finite Markov chain $(B,\tilde p,\tilde \mu)$ is reversible, since its Laplacian $T$ is symmetric. 

In fact, these considerations agree with \cite[Proposition 2.2]{HHW2017} or \cite[Proposition 2]{hassannezhad-miclo20}, stating that the Steklov operator $T$ can be viewed as a symmetric graph Laplacian on $C(B,\R)$ with suitable transition states in the reversible case.

\subsection{Steklov eigenvalues as limit of Laplacian eigenvalues} \label{sec:stekconv}

In the following two subsections, we  are still in the context of non-reversible continuous time Markov chains. We retain the assumption that every vertex in $V$ is connected to some vertex in $B$ via a directed path and refer to this henceforth simply by saying that ``$V$ is connected to $B$''. Our aim is to prove a certain limit behaviour of the eigenvalues of operators $\Delta_r := (1_B + r1_{B^c} )\Delta$ as $ r \to \infty$, which can be viewed as speeding up the Markov chain on the interior vertices $B^c = V \setminus B$. Some of these eigenvalues have the property that their real parts escape to infinity, while the other $b=|B|$ eigenvalues converge to the eigenvalues of the Steklov operator $T$. In fact, we will prove Theorem \ref{thm:convlapsteklov} by establishing Theorem \ref{thm:convInfinity} and Theorem \ref{thm:convFinite} below. 

\subsection{The escaping eigenvalues}

In this subsection, we write, without loss of generality, $V=\{1,2,\dots,n\}$ and $B = \{1,2,\dots,b\}$. Any function $f\in C(V,\mathbb{C}),$ can be identified with the column vector $(f(1),f(2),\cdots, f(n))^\top$. The Laplacian $\Delta$ can be written as the block matrix
\[
  \Delta = 
  \begin{pmatrix}
    L_{11} & L_{12}\\
    L_{21} & L_{22}
  \end{pmatrix},
\] where $L_{11}$ and $L_{22}$ are $b\times b$ and $(n-b)\times(n-b)$ matrices respectively.
Note that $L_{22}$ corresponds to the Laplacian on $C(B^c,\C)$ with Dirichlet boundary condition on $B$. 

For $r>0,$ let $D_r:=\mathrm{diag}(\underbrace{1,\cdots, 1}_b,\underbrace{r,\cdots, r}_{n-b})$. The matrix representation of the rescaled Laplacian is given by
$$\Delta_r = D_r \Delta.$$ 

Let $\{\lambda_1^r,\cdots, \lambda_{n}^r\}$ be the set of eigenvalues of $\Delta_r$, depending continuously on $r$ (see \cite[Theorems II.5.1 and II.5.2]{kato95}).

\begin{proposition}\label{prop:diricheigenvalue}
    Assume that $V$ is connected to $B$, then every eigenvalue of $L_{22}$ has positive real part.
\end{proposition}
\begin{proof}
    Let $L_{22}=\{l_{ij}\}_{(n-b)\times (n-b)}.$ Note that 
    $$l_{ii} 
    \geq \sum_{j \in [n-b] \setminus \{i\}}| l_{ij} |=:R_i,\quad 1\leq i\leq n-b.$$
    By the Gershgorin Circle Theorem (see, e.g., \cite[Theorem 6.1.1]{horn-johnson13}), the eigenvalues of $L_{22}$ are in the union of Gershgorin disks $B_{R_i}(l_{ii})$, which implies that all eigenvalues of $L_{22}$ have non-negative real parts. Moreover, every eigenvalue with vanishing real part must be zero.
    Since $V$ is connected to $B$, Theorem \ref{thm:Dirich-Prob} implies that 
all eigenvalues of $L_{22}$ are nonzero.     
This proves the proposition.
\end{proof}

\begin{theorem}\label{thm:convInfinity}
Assume that $V$ is connected $B$ and let $\lambda_i^r$ be the continuous eigenvalues of $\Delta_r$. Let $\epsilon > 0$ be the minimum of the real parts of the eigenvalues of $L_{22}$. 
Then there exists pairwise distinct $i_1,\dots,i_{n-b} \in [n]$ such that
\begin{equation} \label{eq:liminflambdaijr} 
\liminf_{r \to \infty} \frac{1}{r} \mathrm{Re} (\lambda_{i_j}^r) \ge \epsilon \quad \text{for all $j \in [n-b]$.} 
\end{equation}
In particular, the real parts of these eigenvalues tend to infinity, as $r \to \infty$.
\end{theorem}


\begin{proof}
Since
$$\Delta_r=D_r^{\frac12}D_r^{\frac12}\Delta D_r^{\frac12}D_r^{-\frac12},$$ 
$\Delta_r$ is similar to the matrix 
$D_r^{\frac12}\Delta D_r^{\frac12}.$ Note that
$D_r^{\frac12} \Delta D_r^{\frac12}=r H_r$ for $$H_r = 
  \begin{pmatrix}
    \frac1r L_{11} & \frac{1}{\sqrt{r}}L_{12}\\
   \frac{1}{\sqrt{r}} L_{21} &  L_{22}
  \end{pmatrix}.
$$ As $r\to \infty,$  $$H_r \to 
  \begin{pmatrix}
   0 & 0\\
   0 &  L_{22}
  \end{pmatrix} = H_\infty.
$$ 
Note that the eigenvalues of the limit matrix $H_\infty$ are the eigenvalues of $L_{22}$ together with the eigenvalue $0$ of multiplicity $b$. Recall that the minimum $\epsilon$ of the real parts of the eigenvalues of $L_{22}$ is strictly positive by Proposition \ref{prop:diricheigenvalue}.
Again, by \cite[Theorems II.5.1 and II.5.2]{kato95}, precisely $b$ of the  continuous functions $f_i(r)=\frac{1}{r} \Re(\lambda_i^r)$ 
converge to $0$, while the other functions, indexed by $i_1,\dots,i_{n-b}$, satisfy \eqref{eq:liminflambdaijr}. This finishes the proof.
\end{proof}

\subsection{The convergent eigenvalues}

We start with the following heat semigroup contraction property.

\begin{lemma}\label{lem:semigroupSubCommute2}
  Let $q: V \to \R$. Then we have for all $f \in C(V,\C)$:
  \[
  |e^{-t(\Delta+q)}f| \leq e^{-t(\Delta+q)}|f|.
  \]
\end{lemma}

\begin{proof}
    Let $u_t:=e^{-t(\Delta+q)} f$. We obtain
    \begin{align*}
        \partial_t |u_t|^2 (x) 
        &= -2\Re \left (\overline{u_t(x)} \, (\Delta+q) u_t(x) \right)
        \\&= -2\sum_y p_{xy} 
        \Re\left( \overline{u_t(x)} (u_t(x) - u_t(y)) \right) - 2 q(x) | u_t(x) |^2
        \\& \leq -2|u_t(x)|\sum_y p_{xy} (|u_t(x)| - |u_t(y)|)  - 2 q(x) | u_t(x) |^2
        \\& = -2|u_t(x)|\, (\Delta +q)|u_t|(x),
    \end{align*}
 and hence, $\partial_t^- |u_t| \leq -(\Delta+q) |u_t|$, meaning that $|u_t|$ is a subsolution to the heat equation with added potential $q$. This implies
 \[|u_t| \leq e^{-t(\Delta+q)} |u_0|, \]
 by verifying $\partial_s^- \left( e^{-(t-s)(\Delta+q)} |u_s|\right) \le 0$.
 Replacing $u_t$ by its definition finishes the proof.
\end{proof}



For $r\geq 1$ and let $\Delta_r := (1_B + r1_{B^c} )\Delta$. Moreover, let 
\begin{equation} \label{eq:Rr} 
R_r := (1+\Delta_r)^{-1} = \int_0^\infty e^{-t} e^{-t \Delta_r} dt 
\end{equation}
be the corresponding resolvents.
Let 
\begin{equation} \label{eq:u}
u = (1_B + \Delta)^{-1} 1_{B^c}.
\end{equation}
Let us briefly explain the well-definedness of $u \in C(V,\R)$ under our standing assumption that $V$ is connected to $B$. Let $v \in C(V,\R)$ be in the kernel of $1_B+\Delta$, that is, $(1_B+\Delta)v=0$. Let $x_0 \in V$ be a maximum of $v$. Without loss of generality, we can assume $v(x_0) \ge 0$ (by taking $-v$, if needed). Then we have $\Delta v(x_0)\ge 0$ and $(1_B+\Delta)v(x_0)=0$ implies $\Delta v(x_0) =0$ and either $x_0\in B$ with $v(x_0)=0$ or $x_0 \in B^c$. In case $x_0\in B^c$, notice that $v(y) = v(x_0)$ for all $y \in V$ with $p(x_0,y)>0$. So $v$ is constant for all neighbours of $x_0$.
Repeating this argument, we will eventually
hit a boundary vertex $x \in B$ with $v(x_0)=v(x)$. Since $v$ is maximal at $x$, we have $\Delta v(x) \ge 0$. Together with $(1_B+\Delta)v=0$ and $v(x_0) \ge 0$, this implies that $v(x) = v(x_0) = 0$. A similar argument applies to the minimum. This finishes the proof that $u$ is well-defined. 


Our next aim is to prove $u \ge 0$. Now we choose $x_0 \in V$ where $u$ is minimal. Then we have $\Delta u(x_0) \le 0$. Suppose $u(x_0) < 0$. Note that
$$ 1_B u \ge (1_B+\Delta)u = 1_{B^c} \quad \text{at $x_0$}, $$
which leads to a contradiction in both cases $x_0 \in B$ and $x_0 \in B^c$.

We use $u$ to estimate the resolvents $R_r$.
\begin{lemma} \label{lem:resolventEstimates}
    For $s<r$, and $f\in C(V,\C)$ with $|f|\leq 1_{B^c}$ we have
 \begin{itemize}
     \item[(a)]
     \[
     \|R_s f \|_\infty \leq \frac 1 s \|u\|_{\infty},
     \]
     \item[(b)] 
         \[
\|R_s - R_r \|_{\infty \to \infty} \leq \frac 2 s \|u\|_{\infty},   
    \]
 \end{itemize}
where $u$ is introduced in \eqref{eq:u}.
\end{lemma}

\begin{proof}
We first prove (a).
We observe that Lemma \ref{lem:semigroupSubCommute2} and \eqref{eq:Rr} imply
    \begin{align} \label{eq:Rsf1Wc}
    |R_s f| \leq (1 +  \Delta_s )^{-1} |f| \leq
    (1 +  \Delta_s )^{-1} 1_{B^c}.       
    \end{align}
    Let $u$ be as introduced in \eqref{eq:u}. Then we have on $B^c$     
    \[
    (1+\Delta_s) u  = u +  s \Delta u  \geq s(1_B +  \Delta) u = s1_{B^c} = s,  
    \]
    where we used $u \ge 0$. Moreover, we have on $B$,
    \[
    (1+\Delta_s) u =  (1_B + \Delta)u = 0,
    \]
    giving
    \[
    (1+\Delta_s) u \geq s1_{B^c}.
    \]
    Together with \eqref{eq:Rsf1Wc}, this implies, pointwise at all vertices
    \[
     u \geq (1+\Delta_s)^{-1} (s1_{B^c}) \geq s|R_s f|. 
    \]
   Rearranging proves (a).
   We now prove (b).
    We have
    \begin{align}\label{eq:resolventEquation}
    R_s - R_r = R_s (\Delta_r - \Delta_s) R_r.        
    \end{align}  
    Since 
    $$ \Delta_r - \Delta_s = (r-s) 1_{B^c} \Delta = \frac{r-s}r 1_{B^c}   \Delta_r, $$ 
    we have
    \[
    (\Delta_r - \Delta_s) R_r = 
    \frac{r-s}r 1_{B^c} (1-R_r).
    \]
    As $\|R_r\|_{\infty \to \infty} \leq 1$ by \eqref{eq:Rr} and the fact that $e^{-t\Delta_r}$ is an $L_\infty$-contraction, we find 
    \[
    \|(\Delta_r - \Delta_s) R_r \|_{\infty \to \infty} \leq 2.
    \]    
    We conclude that if $\|g\|_\infty \leq 1$, then, \[|(\Delta_r - \Delta_s) R_r g | \leq 2 \cdot 1_{B^c},\]
    and thus, by applying (a),
    \[
\|R_s (\Delta_r - \Delta_s) R_r g\|_\infty \leq \frac 2 s \|u\|_\infty,
    \]
finishing the proof using    \eqref{eq:resolventEquation}. 
\end{proof}

\begin{lemma} \label{lem:resolvent-convergence} Let $T: C(B,\C) \to C(B,\C)$ be the Steklov operator introduced in Definition \ref{def:steklov}.
There exists an operator $R_\infty : C(V,\C) \to C(V,\C)$
such that
\begin{itemize}
    \item [(a)] $R_s \to R_\infty$ for $s \to \infty$;
    \item [(b)] $R_\infty f=0$ whenever $f|_B = 0$;
    \item [(c)] For all $f \in C(V,\C)$,
    \[
    \Delta R_\infty f  = 1_B(f - R_\infty f); 
    \]
    \item [(d)] For all $f \in C(V,\C)$,
    \[ (R_\infty f) |_B = (1+ T)^{-1} (f|_B);
    \]
    \item [(e)] For all $f \in C(V,\C)$ we have
    $$ R_\infty = {\rm{Ext}} \circ (1+T)^{-1} \circ \iota_B, $$
    where ${\rm{Ext}}$ is the $\Delta$-harmonic extension operator of functions in $C(B,\C)$ to $C(V,\C)$ and $\iota_B$ is the restriction operator from $C(V,\C)$ to $C(B,\C)$.
\end{itemize}
\end{lemma}

\begin{proof}
    By Lemma~\ref{lem:resolventEstimates}(b),
    $R_s$ is a Cauchy sequence as $s \to \infty$, and assertion (a) follows easily.
    Assertion (b) follows from Lemma~\ref{lem:resolventEstimates}(a).
    For (c), we observe that on $B$,
    \[
    \Delta R_s f = \Delta_s R_s f
    = f - R_s f,
    \]
    and on $B^c$,
    \[
    \Delta R_s f = \frac 1 s \Delta_s R_s f \to 0 \quad \text{as $s \to \infty$.}
    \]
Taking the limit and combining these two observation proves (c).    
We finally prove (d).
We notice that by $(c)$, the function $R_\infty f$ is the unique harmonic extension of $(R_\infty f)|_B$ and thus by the definition of the Steklov operator $T$, we obtain on $B$,
\[
T ((R_\infty f)|_B) = \Delta R_\infty f.
\]
Hence, on $B$ (applying (c) in the last equality),
\[
(1+T) ((R_\infty f)|_B) = (1+\Delta) R_\infty f = f|_B,
\]
and (d) follows by applying $(1+T)^{-1}$.
Note that (e) is a direct consequence of (c) and (d).

Thus, the proof of the lemma is finished.
\end{proof}

\begin{theorem}\label{thm:convFinite}
Assume that $V$ is connected to $B$. 
We have the following eigenvalue convergence of the operators $\Delta_r = D_r \Delta: C(V,\C) \to C(V,\C)$ to the Steklov operator $T: C(B,\C) \to C(B,\C)$: The spectral measure 
$$ \mu_r := \sum_{i=1}^{n} \delta_{\lambda_i^r} $$
of $\Delta_r$ with eigenvalues $\lambda_i^r$ converges to the spectal measure 
$$ \mu_\infty := \sum_{i=1}^{b} \delta_{\sigma_i} $$
of $T$ with eigenvalues $\sigma_i$,
as $r \to \infty$, in the measure convergence, that is, we have for all $f \in C_c(\C)$,
$$ \mu_r(f) \to \mu_\infty(f). $$
\end{theorem}

\begin{proof}
  By Lemma \ref{lem:resolvent-convergence}(a), we have
  $$ R_s \to R_\infty $$
  in the operator sense, and $R_\infty$ has the following block matrix structure:
  $$ R_\infty = \begin{pmatrix} (1 + T)^{-1} & 0 \\ \iota_{B^c}\circ {\rm{Ext}}\circ (1+T)^{-1} & 0 \end{pmatrix}, $$
  by Lemma \ref{lem:resolvent-convergence}(e). Therefore, using again \cite[Theorems II.5.1 and II.5.2]{kato95}, we have the measure convergence
  $$ \sum_{i=1}^{n} \delta_{(1+\lambda_i^r)^{-1}} \to (n-b) \cdot \delta_0 + \sum_{i=1}^{b} \delta_{(1+\sigma_i)^{-1}}, $$
  which implies the statement of the theorem.
\end{proof}

\bigskip


\section*{Acknowledgements}
We express our sincere gratitude to Prof.\ Matthias Keller for valuable and insightful discussions that contributed significantly to the proof of one of the main results, Theorem~\ref{t:gcheeger}.

B. Hua is supported by NSFC, no.12371056. S. Kamtue is supported by grants for development of new faculty staff, Ratchadaphiseksomphot Fund, Chulalongkorn University. S. Liu is supported by NSFC, no. 12431004.

\printbibliography

\end{document}